\documentclass{article}
\usepackage[left=3.18cm,right=3.18cm,top=2.54cm,bottom=2.54cm]{geometry}
\usepackage{indentfirst}
\usepackage{amsfonts}
\usepackage{amsmath}
\usepackage{amssymb}
\usepackage{amsthm}
\usepackage{enumerate}
\usepackage{hyperref}
\usepackage{cleveref}
\usepackage{tikz,tikz-cd}
\usepackage{extpfeil}
\usepackage{setspace}
\usepackage{xcolor}

\newcommand\QQ{\mathbf{Q}}

\newcommand\RR{\mathbf{R}}
\newcommand\NN{\mathbf{N}}
\newcommand\ZZ{\mathbf{Z}}
\numberwithin{equation}{section}
\newcommand\n{\mathbf{n}}

\theoremstyle{definition}

\newtheorem{theorem}{Theorem}[section]
\newtheorem{definition}[theorem]{Definition}
\newtheorem{proposition}[theorem]{Proposition}

\newtheorem{assumption}{Assumption}
\newtheorem{corollary}[theorem]{Corollary}
\newtheorem{lemma}[theorem]{Lemma}

\newtheorem{remark}[theorem]{Remark}
\newtheorem{example}[theorem]{Example}

\DeclareMathOperator{\Tor}{Tor}
\DeclareMathOperator{\im}{im}
\DeclareMathOperator{\Hom}{Hom}
\DeclareMathOperator{\deR}{deR}
\DeclareMathOperator{\GL}{GL}

\begin{document}
\title{\textbf{Characterization of differential K-theory by the hexagon diagram}}
\author{Jiahao Hu}
\date{}
\maketitle
\begin{abstract}
   Using a canonical topology on differential K-theory induced from the Frech\'et space topology on differential forms and the discrete topology on topological K-theory, we prove that differential K-theory is uniquely determined by the character diagram up to a unique natural equivalence, thus giving an affirmative answer to a question asked by Simons and Sullivan in \cite{SS10}. We further deduce rigidity results including that there is a unique way of realizing $\RR/\ZZ$-K-theory as the flat theory, strengthening the results of \cite{BS10}.
\end{abstract}

\section{Introduction}
The purpose of this paper is to prove an axiomatic characterization theorem for differential K-theory similar to the one obtained by Simons and Sullivan in \cite{SS08} for differential cohomology; they showed that differential cohomology as a functor fitting into certain hexagon diagram is unique up to a unique natural equivalence. In a later work \cite{SS10} they further proved that differential K-theory $\hat{K}$ fits into a similar hexagon diagram below, called the \textbf{character diagram}, with exact boundaries and diagonals; naturally they asked whether or not such diagram axiomatically characterizes differential K-theory. 

\begin{center}
\setlength{\unitlength}{0.4cm}
\begin{picture}(24,16)
\put(5.5,1.5){$0$} 
\put(18.5,1.5){$0$} 
\put(5.5,14.5){$0$} 
\put(18.5,14.5){$0$} 
\put(7.5,4){$\Omega^{od}/\Omega_U$} 
\put(15.5,4){$\Omega_{BU}$} 
\put(3.6,8){$H^{od}(\RR)$} 
\put(11.8,8){$\hat{K}$} 
\put(18.5,8){$H^{ev}(\RR)$} 
\put(7.5,12){$K^{-1}_{\RR/\ZZ}$} 
\put(15.8,12){$K$} 
\put(10.5,12.2){\vector(1,0){4}} 
\put(10.8,4.2){\vector(1,0){3.7}} 
\put(5.5,9){\vector(1,1){2.3}} 
\put(5.5,7.5){\vector(1,-1){2.2}} 
\put(9.3,11.2){\vector(1,-1){2.2}} 
\put(9.3,5.2){\vector(1,1){2.3}} 
\put(13,9){\vector(1,1){2.3}} 
\put(13,7.5){\vector(1,-1){2.3}} 
\put(16.8,11.3){\vector(1,-1){2.3}} 
\put(16.8,5.1){\vector(1,1){2.3}} 
\put(6.2,2.2){\vector(1,1){1.5}} 
\put(6.2,14.2){\vector(1,-1){1.3}} 
\put(16.8,12.8){\vector(1,1){1.5}} 
\put(16.9,3.6){\vector(1,-1){1.4}} 
\put(12,4.5){\small{$d$}}
\put(12,12.5){\small{$\beta$}}
\put(9.5,6.5){\small{$i$}}
\put(11,10.2){\small{$j$}}
\put(13,10){\small{$\delta$}}
\put(14.3,6.5){\small{$ch$}}
\put(5,5.6){\small{$\deR$}} 
\put(18.2,5.6){\small{$\deR$}} 
\put(3.8,10.4){\small{mod $\ZZ$}}
\put(18.2,10.2){\small{$c$}}
\put(2,6){\vector(1,1){1.4}} 
\put(2,10.5){\vector(1,-1){1.4}} 
\put(21,9){\vector(1,1){1.4}}
\put(21,7.2){\vector(1,-1){1.4}}
\end{picture}
\end{center}
The sequence along the upper boundary in the above character diagram
\[
\to H^{od}(\RR)\xrightarrow{\text{mod } \ZZ} K^{-1}_{\RR/\ZZ}\xrightarrow{\beta} K\xrightarrow{c} H^{ev}(\RR)\to 
\]
is identified, via Chern character (tensored with $\RR$), with the Bockstein exact sequence for complex K-theory associated to the coefficient exact sequence $0\to \ZZ\to \RR\to \RR/\ZZ\to 0$, and the sequence along the lower boundary
\[
\to H^{od}(\RR)\xrightarrow{deR} \Omega^{od}/\Omega_U\xrightarrow{d} \Omega_{BU}\xrightarrow{deR} H^{ev}(\RR)\to 
\]
comes from de Rham theory of representing real cohomology classes by differential forms.  The groups and morphisms in this diagram will be explained in details in \Cref{sec3}.

Our main result is that the character diagram indeed determines differential K-theory.
\begin{theorem}[Uniqueness]\label{maintheorem}
Any two differential K-functors $(\hat{K}',i',j',\delta',ch')$ and $(\hat{K},i,j,\delta,ch)$ are naturally equivalent via a natural transformation $\Phi:(\hat{K}',i',j',\delta',ch')\to (\hat{K},i,j,\delta,ch)$; such $\Phi$ is unique.
\end{theorem}
Here, by a \textbf{differential K-functor} we mean a $5$-tuple $(\hat{K},i,j,\delta,ch)$, consisting of a contravariant functor $\hat{K}$ from the category of \textit{compact manifolds with corners} (with smooth mappings) to the category of abelian groups, and natural transformations $(i,j,\delta,ch)$ fitting into the character diagram with exact diagonals; by a natural transformation from a differential K-functor $(\hat{K}',i',j',\delta',ch')$ to another $(\hat{K},i,j,\delta,ch)$ we mean a natural transformation $\Phi:\hat{K}'\to\hat{K}$ such that $\Phi\circ i=i'$, $\Phi\circ j'=j$, $\delta\circ\Phi=\delta'$ and $ch\circ\Phi=ch'$.

Prior to our work here, the question of axiomatically characterizing differential K-theory and other differential refinements of generalized cohomology theories has been throughly analyzed by Bunke and Schick \cite{BS10} under a different set of axioms from which uniqueness results are obtained. The major difference between their axioms and the character diagram is the absence of the group $K^{-1}_{\RR/\ZZ}$ in their axioms; and their analysis does not directly answer Simons and Sullivan's question\footnote{More recently Ishan Mata \cite{MI21} obtained a partial affirmative answer by employing homological algebra.}.

Nevertheless, their techniques can be appropriately modified to produce between any two given functors fitting into the above character diagram a natural transformation compatible with all the maps except for $K^{-1}_{\RR/\ZZ}\to\hat{K}$. The deviation from this compatibility would roughly be a natural transformation $K^{-1}_{\RR/\ZZ}\to H^{od}(\RR)\cong K^{-1}_{\RR}$, which is not a priori zero because there are non-trivial homomorphisms $\RR/\ZZ\to \RR$.

To prove this compatibility, we show that there are canonical and natural topologies on all the groups in the character diagram; differential forms carry Fr\'echet space topology, topological K-theory carries discrete topology, and all other groups inherit topologies from these two. The desired compatibility will be proved in the spirit of the fact that there is no non-trivial \textit{continuous} homomorphisms from $\RR/\ZZ$ into $\RR$; our uniqueness theorem will then follow at once.

The topology on $\hat{K}$ is expected, for example from the structured bundle model of Simons and Sullivan \cite{SS10}, and the idea of a topology on differential K-theory is certainly around among experts; here we make it precise without referring to any particular model. Moreover, we shall prove that this topology is strong enough to functorially recover all the surrounding groups in the character diagram; the topological K-theory is the group of connected components of $\hat{K}$ and $K^{-1}_{\RR/\ZZ}$ is the closure of the torsion subgroup of $\hat{K}$. It should also be clear from our argument that the same holds for other differential refinements of generalized cohomology theories, so we believe it is more appropriate to study differential cohomology theories as functors valued in the category of abelian \textit{topological} groups; we hope this can be a useful insight for future studies.

Compared to Bunke and Schick's approach, ours has the extra benefit of showing that there is a unique way of realizing $K^{-1}_{\RR/\ZZ}$ as the flat theory (the kernel of the curvature map $ch: \hat{K}\to \Omega_{BU}$), see \Cref{rigidity}; this result is not deducible from Bunke and Schick's work in \cite{BS10}\footnote{as stated in \cite[Theorem 7.12]{BS10}.} and answers a folklore question. 

We should point out that our argument cannot prove that \textit{odd} differential K-theory is unique under the character diagram axiom. The counter example given in \cite{BS10} continues to hold so we shall stick to \textit{even} differential K-theory throughout.
\subsection*{Organization}
We organize this paper by starting with a general discussion on abelian topological groups; in \Cref{sec2} we functorially build hexagon diagrams for suitable abelian topological groups. In \Cref{sec3} we topologize the character diagram and prove it is naturally isomorphic to the hexagon diagram for the central group $\hat{K}$. \Cref{sec4} is devoted to adapting the technical tools of \cite{BS10} to our situation. We prove our uniqueness theorem in \Cref{sec5} and conclude with two rigidity results in \Cref{sec6}; one generalizes \cite[Lemma 1.1]{SS08} and the other concerns realizing $K^{-1}_{\RR/\ZZ}$ as the flat theory.

\subsection*{Acknowledgement}
The author thanks Dennis Sullivan for suggesting this problem to him and for fruitful discussions, Runjie Hu for useful comments on an early draft of this paper, and anonymous referee for constructive advices.

\section{Hexagon diagram for abelian topological groups}\label{sec2}
Throughout, subgroups, quotient groups and product groups have subspace, quotient, and product topologies respectively. The field of real numbers $\RR$ is equipped with its standard topology.
\subsection{Strict homomorphism} Recall that every group homomorphism $\phi:G\to H$ yields a group isomorphism $\bar{\phi}:G/\ker\phi\xrightarrow{\cong} \im \phi$. In contrast, if $G,H$ are \textit{topological} groups and $\phi$ is a \textit{continuous} homomorphism, then we cannot conclude $\bar{\phi}$ is an isomorphism of topological groups: a continuous bijection need not be a homeomorphism. This leads to the following notion.
\begin{definition}
	A continuous homomorphism between topological groups is said to be \textbf{strict} if it is open onto its image, i.e. it takes open sets to open subsets of its image.
\end{definition}

\begin{lemma}\label{lem:strictprop}
	\begin{enumerate}[(i)]
		\item Let $N$ be a normal subgroup of a topological group $G$. Then the quotient map $\pi: G\to G/N$ is open. In particular $\pi$ is strict.
		\item Let $\pi: G\to G/N$ be a quotient of topological groups. Then a continuous homomorphism $\phi:G/N\to H$ is strict if and only if $\phi\circ\pi: G\to H$ is strict. 
		\item 	A continuous homomorphism $\phi:G\to H$ is strict if and only if it induces an isomorphism of topological groups $\overline{\phi}:G/\ker\phi\xrightarrow{\cong} \im\phi$.
	\end{enumerate}
\end{lemma}
\begin{proof}
	\begin{enumerate}[(i)]
		\item Let $U$ be an open subset of $G$. Since $\pi^{-1}(\pi(U))=\bigcup\limits_{g\in N} gU$ is open in $G$, by definition of quotient topology $\pi(U)$ is open in $G/N$.
		\item The `only if' part follows from the openness of $\pi$ guaranteed by (i). Now assume $\phi\circ\pi$ is strict and $U$ is an open subset of $G/N$. Then by continuity of $\pi$, we see
	$\phi(U)=\phi\circ\pi(\pi^{-1}(U))$ is open in $\im(\phi\circ\pi)=\im\phi$. Therefore $\phi$ is strict. This proves the `if' part.
		\item Consider the canonical factorization $\phi: G\xrightarrow{\pi}G/\ker\phi\xrightarrow{\bar{\phi}}\im\phi\xrightarrow{\iota} H$.
	Then
	\begin{align*}
		\phi\text{ is strict}&\Leftrightarrow \bar{\phi}\circ\pi\text{ is strict} \quad\text{(by definition)}\\
		&\Leftrightarrow\bar{\phi}\text{ is strict} \quad\text{(by (ii))}\\
		&\Leftrightarrow\bar{\phi}\text{ is an isomorphism of topological groups.}
	\end{align*}
	\end{enumerate}
\end{proof}
Compositions of strict homomorphisms need not be strict. For example, let $\iota_A: A\to G$ be an embedding of topological groups and $\pi_B: G\to G/B$ be a quotient by a normal subgroup $B$. Both $\iota_A$ and $\pi_B$ are clearly strict, but their composition $\pi_B\circ\iota_A$ is not strict in general, for instance take $G=\RR^2$, $B=\ZZ^2$ and $A$ to be a line passing through the origin with irrational slope.

 Here we state an easy yet useful sufficient condition that ensures the strictness of $\pi_B\circ \iota_A$.
 
\begin{lemma}\label{lem:secondiso} If either $A$ or $B$ is an open subgroup of $G$, then the composition $\pi_B\circ\iota_A: A\to G/B$ is open. In particular, $\pi_B\circ\iota_A$ is strict and induces an isomorphism of topological groups $A/A\cap B\cong AB/B$.
\end{lemma}
\begin{proof}
	If $A$ is open, then both $\iota_A$ and $\pi_B$ are open, hence $\pi_B\circ\iota_A$ is open. If $B$ is open, then since open subgroups are also closed, $B$ is clopen (i.e. closed and open) in $G$. Recall that the quotient of a topological group by a clopen normal subgroup is discrete, so $G/B$ is discrete and therefore all maps into $G/B$ are open.
\end{proof}

\begin{definition}
	A sequence of continuous homomorphisms among topological groups is \textbf{strictly exact} if it is an exact sequence of groups in which all the homomorphisms are strict.
\end{definition}

\begin{example}\label{example:twosubgrps}\footnote{This example is inspired by the work of Ishan Mata in \cite{MI21}, and should be compared to Lawson spark.} Let $A,B$ be normal subgroups of a topological group $G$. Assume either $A$ or $B$ is open, then it follows from \Cref{lem:secondiso} that the following commutative diagram has strictly exact rows and columns.
\begin{center}
\begin{tikzcd}
            & 1     \arrow[d]           & 1 \arrow[d]             & 1 \arrow[d]  &   \\
1 \arrow[r] & {A\cap B} \arrow[r] \arrow[d] & {B} \arrow[r] \arrow[d] & {B/A\cap B} \arrow[r] \arrow[d] & 1 \\
1 \arrow[r] & {A} \arrow[r] \arrow[d] & {G} \arrow[r] \arrow[d] & {G/A} \arrow[r] \arrow[d] & 1 \\
1 \arrow[r] & {A/A\cap B} \arrow[r] \arrow[d] & {G/B}\arrow[r] \arrow[d] & {G/AB} \arrow[r] \arrow[d] & 1 \\
            & 1                         & 1                       & 1           &  
\end{tikzcd}
\end{center}
\end{example}

\subsection{Hexagon diagram}
We shall now derive from \Cref{example:twosubgrps} a hexagon diagram for a suitable \textit{abelian} topological group $G$  with respect to its two \textit{closed} subgroups: the identity connected component $G_e$, and the \textit{closure} of torsion $T=\overline{\Tor}(G)$. Technical assumptions on $G$ are needed.
\begin{assumption}\label{A1}
	Assume that $G$ is a locally connected, $G_e\cap T$ is connected and $G/(G_e+T)$ is torsion-free.
\end{assumption}
 
Recall that the quotient of a topological group by a closed subgroup is Hausdorff. Thus $\pi_0 G=G/G_e$ and $G/T$ are both Hausdorff. Further, since $G$ is assumed to be locally connected, $G_e$ is not only closed but also open in $G$, therfore the condition of \Cref{example:twosubgrps} is met. Moreover since $G_e$ is clopen, we see that $\pi_0 G$ is discrete.

\begin{lemma} Under \cref{A1}, we have canonical isomorphisms and equalities of abelian topological groups:
	\begin{enumerate}[(i)]
		\item $G_e\cap T=\overline{\Tor}(G_e)=T_e$
		\item $G_e/(G_e\cap T)\cong (G_e+T)/T=(G/T)_e$
		\item $T/(G_e \cap T)\cong(G_e+T)/G_e=\Tor(\pi_0 G)$
	\end{enumerate}
\end{lemma}
\begin{proof}
	\begin{enumerate}[(i)]
		\item $G_e$ being clopen in $G$ implies that $G_e\cap T$ is clopen in $T$. This means $G_e\cap T$ is a union of connected components of $T$. But $G_e\cap T$ is assumed to be connected, therefore $G_e\cap T=T_e$. The second equality follows from $\Tor(G_e)=G_e\cap \Tor(G)$ by taking closure and noticing that $G_e$ is clopen. 
		\item The first isomorphism follows from \Cref{lem:secondiso}. Then the connectedness of $G_e$ implies that both $G_e/(G_e\cap T)$ and $(G_e+T)/T$ are connected. We claim that $(G_e+T)/T$ is clopen in $G/T$, then $(G_e+T)/T=(G/T)_e$ follows immediately from the connectedness of $(G_e+T)/T$. To prove the claim, notice that since $\pi_0 G$ is discrete, its quotient $G/(G_e+T)$ must also be discrete. Therefore $(G_e+T)/T$, being the kernel of $G/T\to G/(G_e+T)$, is clopen in $G/T$.
		\item Again the first isomorphism follows from \Cref{lem:secondiso}. Now observe that $(T+G_e)/G_e$ is the image of $T$ under the projection $G\to\pi_0 G$. Since $\Tor(G)$ is dense in $T$, and group homomorphisms take torsion elements to torsion elements, we have $(T+G_e)/G_e\subseteq \overline{\Tor}(\pi_0 G)=\Tor(\pi_0 G)$ (recall that $\pi_0 G$ is discrete). On the other hand, since $G/(G_e+T)$ is assumed to be torsion-free, we see that $(G_e+T)/G_e$, being the kernel of $\pi_0 G\to G/(G_e+T)$, must be contained in $\Tor(\pi_0 G)$. This proves $(G_e+T)/G_e=\Tor(\pi_0 G)$.
	\end{enumerate}
\end{proof}

\begin{corollary}\label{lem:sqrdiag}
Under \cref{A1}, we have the following commutative diagram with strictly exact rows and columns.
\begin{center}
\begin{tikzcd}
            & 0     \arrow[d]           & 0 \arrow[d]             & 0 \arrow[d]  &   \\
0 \arrow[r] & {T_e} \arrow[r] \arrow[d] & {T} \arrow[r] \arrow[d] & {\pi_0 T} \arrow[r] \arrow[d] & 0 \\
0 \arrow[r] & {G_e} \arrow[r] \arrow[d] & {G} \arrow[r] \arrow[d] & {\pi_0 G} \arrow[r] \arrow[d] & 0 \\
0 \arrow[r] & {G_e/T_e} \arrow[r] \arrow[d] & {G/T}\arrow[r] \arrow[d] & {\pi_0(G/T)} \arrow[r] \arrow[d] & 0 \\
            & 0                         & 0                       & 0           &  
\end{tikzcd}
\end{center}
\qed
\end{corollary}

Now if we rotate the above diagram counter-clockwise by 45 degrees, thus placing $T_e$ and $\pi_0(G/T)$ on the far left and far right respectively, and also draw an arrow connecting $T$ to $\pi_0 G$ standing for the composition, and similarly connect $G_e$ to $G/T$, then a hexagon is formed, centered around the group $G$. This diagram now looks similar to the character diagram. However note that in the character diagram, the groups on the far left and right are real vector spaces. For this reason, we consider two further constructions: the universal cover of $T_e$ and  $\pi_0 G\otimes_\ZZ\RR$. Of course, extra assumptions have to be made.

\begin{assumption}\label{A2}
	Assume that $T_e$ is connected, locally path-connected and semilocally simply connected so that it has a universal cover $\widetilde{T_e}$. Assume that $\pi_0(G/T)$ is finitely generated.
\end{assumption} 
The universal covering map $\widetilde{T_e}\to T_e$, precomposed with the inclusions $T_e\to T$ and $T_e\to G_e$, gives rise to two strict homomorphisms
\[
\widetilde{T_e}\to T,\text{ and } \widetilde{T_e}\to G_e.
\]

Next we consider $\pi_0 G\otimes_\ZZ\RR$. Since $\pi_0 T$ is torsion, the exact sequence of discrete groups $0\to \pi_0 T\to \pi_0 G\to\pi_0 (G/T)\to 0$ gives an isomorphism $\pi_0 G\otimes\RR\cong \pi_0 (G/T)\otimes \RR$. Moreover, since $\pi_0(G/T)\cong \pi_0 G/\pi_0 T\cong \pi_0 G/\Tor(\pi_0 G)$ is torsion free, we see that $\pi_0 (G/T)\hookrightarrow \pi_0 (G/T)\otimes \RR\cong\pi_0 G\otimes\RR$.  Now that $\pi_0(G/T)$ is assumed to be finitely generated, the group $\pi_0 G\otimes \RR\cong\pi_0(G/T)\otimes \RR$ is a finite dimensional $\RR$-vector space that inherits a topology from $\RR$, in which $\pi_0(G/T)$ is identified as a discrete lattice.
So the inclusion $\pi_0 (G/T)\hookrightarrow \pi_0 G\otimes \RR$, composed with the quotient maps $\pi_0 G\to \pi_0(G/T)$ and $G/T\to \pi_0(G/T)$, gives rise to two strict homomorphisms
\[
\pi_0 G\to \pi_0 G\otimes\RR,\text{ and } G/T\to \pi_0 G\otimes\RR.
\]

\begin{remark}
	It is necessary to assume that $\pi_0(G/T)$ is finitely generated, for in general the inclusion $\pi_0(G/T)\hookrightarrow \pi_0(G)\otimes_\ZZ\RR$ is not an embedding of topological groups. For example, take $\pi_0 G=\pi_0(G/T)=\QQ$ with discrete topology, then the above inclusion $\QQ\hookrightarrow\QQ\otimes_\ZZ\RR\cong\RR$ is the usual inclusion of $\QQ$ into $\RR$. But the subspace topology on $\QQ$ is not discrete.
\end{remark}

\begin{proposition}\label{prop:hexforgp}
Under \cref{A1} and \cref{A2}, the boundary and diagonal exact sequences in the following commutative diagram are strictly exact.
\begin{center}
\setlength{\unitlength}{0.35cm}
\begin{picture}(24,16)
\put(5.5,1.5){$0$} 
\put(18.5,1.5){$0$} 
\put(5.5,14.5){$0$} 
\put(18.5,14.5){$0$} 
\put(7.8,4){$G_e$} 
\put(15.5,4){$G/T$} 
\put(3.6,8){$\widetilde{T_e}$} 
\put(11.8,8){$G$} 
\put(18,8){$\pi_0 G\otimes_\ZZ\RR$} 
\put(7.8,12){$T$} 
\put(15.5,12){$\pi_0 G$} 
\put(10,12.3){\vector(1,0){4.6}} 
\put(10,4.3){\vector(1,0){4.6}} 
\put(5.2,9){\vector(1,1){2.3}} 
\put(5.2,7.5){\vector(1,-1){2.2}} 
\put(9.3,11.2){\vector(1,-1){2.2}} 
\put(9.3,5.2){\vector(1,1){2.3}} 
\put(13,9){\vector(1,1){2.3}} 
\put(13,7.5){\vector(1,-1){2.3}} 
\put(16.8,11.3){\vector(1,-1){2.3}} 
\put(16.8,5.1){\vector(1,1){2.3}} 
\put(6.2,2.2){\vector(1,1){1.5}} 
\put(6.2,14.2){\vector(1,-1){1.3}} 
\put(16.8,12.8){\vector(1,1){1.5}} 
\put(16.9,3.6){\vector(1,-1){1.4}} 
\end{picture}
\end{center}
\qed
\end{proposition}
\begin{definition}
	The diagram in \Cref{prop:hexforgp} is called the \textbf{hexagon diagram} for $G$, provided $G$ satisfies \cref{A1} and \cref{A2}. The hexagon diagram is clearly functorial in $G$.
\end{definition}

\section{Topology on the character diagram} \label{sec3}
The goal of this section is to show that the smooth structure on a manifold induces natural topologies on the groups in the character diagram so that all the homomorphisms in the character diagram are continuous and strict. Moreover, with such topology the character diagram is canonically isomorphic to the hexagon diagram for $\hat{K}$.

In the character diagram, the sequence along the upper boundary in the above character diagram
\[
\to H^{od}(\RR)\xrightarrow{\text{mod } \ZZ} K^{-1}_{\RR/\ZZ}\xrightarrow{\beta} K\xrightarrow{c} H^{ev}(\RR)\to 
\]
is identified, via Chern character (tensored with $\RR$), with the Bockstein exact sequence for complex K-theory associated to the coefficient exact sequence $0\to \ZZ\to \RR\to \RR/\ZZ\to 0$, and the sequence along the lower boundary
\[
\to H^{od}(\RR)\xrightarrow{deR} \Omega^{od}/\Omega_U\xrightarrow{d} \Omega_{BU}\xrightarrow{deR} H^{ev}(\RR)\to 
\]
comes from de Rham theory of representing real cohomology classes by differential forms.  The groups and morphisms in this diagram will be explained along the way.

Now let $M$ be a compact smooth manifold with corners. We now proceed to topologize each of the groups in the character diagram for $M$. Along the way, we will prove all the homomorphisms in the character diagram are continuous and strict. Furthermore, we will identify $\Omega_{BU}$ and $K^{-1}_{\RR/\ZZ}$ with the component of identity of $\hat{K}$ and the closure of torsion of $\hat{K}$ respectively.

All functors in this section are applied to $M$ unless otherwise stated. For simplicity, we shall sometimes drop $M$ from our notation.
 
\subsection{Differential forms} The space of (real-valued smooth) differential $k$-forms $\Omega^k(M)$ has a Fr\'echet space topology as follows. Choose a Riemannian metric and a connection on $M$. If $\omega$ is a $k$-form, denote its $j^\text{th}$ covariant derivative by $D^j\omega$. Then
\[
\|\omega\|_{n}=\sum_{j=0}^n \sup|D^j \omega|
\]
(where $|\cdot|$ is induced by the Riemannian metric) is a family of seminorms making $\Omega^k(M)$ into a Fr\'echet space. Since $M$ is compact, different choices of metrics and connections yield bounded changes of the seminorms $\|\cdot\|_n$, so the topology on $\Omega^k(M)$ is independent of such choices. If $f: N\to M$ is a smooth map from a compact manifold $N$ into $M$, then by the compactness of $M,N$ and standard estimates, the pull-back homomorphism
\[
f^*: \Omega^k(M)\to\Omega^k(N)
\]
is a bounded linear operator. Similarly by standard estimates \footnote{$\|d\omega\|_n$ is controlled by $\|\omega\|_{n+1}$.} the exterior differential
\[
d:\Omega^k(M)\to\Omega^{k+1}(M)
\]
is a bounded linear operator.

%
%
%
%
\begin{definition} An \textbf{oriented smooth cycle} in $M$ is a pair $(V,f)$ consisting of a closed (i.e. compact without boundary) oriented smooth manifold $V$ and a smooth mapping $f: V\to M$.
\end{definition}
\begin{proposition}\label{exactforms}
A closed differential form $\omega$ on $M$ is exact if and only if $\int_V f^*\omega=0$ for all oriented smooth cycles $(V,f)$.
\end{proposition}
\begin{proof}
The `only if' part follows from Stokes theorem. The `if' part follows from de Rham's theorem that integration induces an isomorphism $H^k_{dR}(M)\cong \Hom(H_k(M;\ZZ),\RR)$ and from \cite{Th54} that $\ZZ$-homology classes, after multiplied by some positive integer (depending only on $k$), can be represented by oriented smooth cycles.
\end{proof}

\begin{corollary}\label{closedsubspace}
$\Omega_{cl}^k=\{\text{closed $k$-forms}\}$ and $d\Omega^{k-1}=\{\text{exact $k$-forms}\}$ are closed vector subspaces of $\Omega^k(M)$. In particular, $\Omega_{cl}^k$ and $d\Omega^{k-1}$ are Fr\'echet spaces.
\end{corollary}
\begin{proof}
Since $d$ is bounded, $\Omega_{cl}^k=\ker d$ is closed. For each oriented smooth $k$-dimensional cycle $(V,f)$, $\omega\mapsto \int_V f^*\omega$ is a bounded linear functional on $\Omega_{cl}^k$. By \Cref{exactforms}, $d\Omega^{k-1}$ is the intersection of the kernels of these functionals, hence $d\Omega^{k-1}$ is closed. The second assertion follows from that closed vector subspaces of Fr\'echet spaces are Fr\'echet.
\end{proof}

\subsection{De Rham and singular cohomology}
 The de Rham cohomology groups $H^*_{dR}(M)=\Omega_{cl}^*/d\Omega^{*-1}$ have induced topologies as sub-quotients of differential forms. Since quotients of Fr\'echet spaces by closed vector subspaces are Fr\'echet, by \Cref{closedsubspace} $H^k_{dR}(M)$ is a Fr\'echet space.

The singular cohomology with $\RR$-coefficients $H^*(M;\RR)$ can be topologized as follows. The $k$-th singular homology $H_k(M;\ZZ)$ is a finitely generated abelian group, on which we equip the discrete topology. Next we equip $C(H_k(M;\ZZ),\RR)=\RR^{H_k(M;\ZZ)}$, the set of continuous functions on $H_k(M;\ZZ)$, with the compact-open topology, or equivalently the product topology on $\RR^{H_k(M;\ZZ)}$; this way $C(H_k(M;\ZZ),\RR)$ is a (Hausdorff) topological vector space. Then through the isomorphism $$H^k(M;\RR)\cong \Hom(H_k(M;\ZZ),\RR)\subseteq C(H_k(M;\ZZ),\RR),$$ we can now endow $H^k(M;\RR)$ with the subspace topology, making it into a finite dimensional topological vector space. We can similarly topologize $H^k(M;\QQ)$ and it is not hard to see that $H^k(M;\QQ)$ is dense in $H^k(M;\RR)$. Indeed, choose a basis $e_i$ for the free part of $H_k(M;\ZZ)$, then a real cohomology class $l\in H^k(M;\RR)$ is determined by the real numbers $l(e_i)$. The class $l$ belongs to $H^k(M;\QQ)$ if and only if $l(e_i)\in \QQ$ for all $i$.

Now that both $H^*_{dR}$ and $H^*(\RR)$ are finite dimensional topological vector spaces, the de Rham isomorphism $H^*_{dR}\cong H^*(\RR)$ must also be a homeomorphism. We shall therefore not distinguish $H^*_{dR}(M)$ from $H^*(M;\RR)$ from now on. 

However, it is probably worth pointing out that the topology we put on $H^k(\RR)$ is clearly functorial with respect to continuous maps, while the topology on $H^k_{dR}$ is functorial with respect to smooth maps, for the de Rham groups and the de Rham isomorphisms depend a priori on the smooth structure.

\subsection{Unitary forms}
\begin{definition}
A \textbf{smooth SAC cycle} $(V,f)$ in $M$ is a closed stably almost complex (SAC) manifold $V$ together with a smooth mapping $f: V\to M$. The \textbf{period} of a closed form $\omega$ over $(V,f)$ is
\[
\int_V f^*\omega\cdot Td(V)
\]
where $Td(V)$ is the Todd class of $V$, which may be represented by a total even closed form by enriching the stable tangent bundle of $V$ with a unitary connection.
\end{definition}
Since (according to Chern-Weil theory) different choices of connections result in cohomologous differential form representatives of $Td(V)$, and since $\omega$ is closed, Stokes theorem ensures the period of $\omega$ over $(V,f)$ is independent of choices of connections on $V$. 

Recall that by definition \cite{SS10} $\Omega_{BU}(M)$ (resp. $\Omega_U(M)$) is the group of closed even (resp. odd) forms on $M$ having periods in $\ZZ$ over all even (resp. odd) dimensional smooth SAC cycles. It is proved in \cite{SS10, SS18} that $\Omega_{BU}$ contains exactly those closed even forms cohomologous to Chern characters of unitary vector bundles on $M$. Similarly, $\Omega_U$ consists of those closed odd forms cohomologous to pull-backs by maps of $M$ into the unitary group $U$ (i.e. union of $U(n)$) of the transgressed Chern character. We therefore refer to them as \textbf{unitary forms}. The spaces of unitary forms $\Omega_U(M)$ and $\Omega_{BU}(M)$, as subspaces of differential forms, are equipped with subspace topologies.

By Stokes theorem, exact forms have vanishing periods over smooth SAC cycles, so in particular exact forms are unitary forms. Also notice that given any enriched SAC cycle $(V,f)$, the map
\[
\Omega_{cl}\to \RR,\quad \omega\mapsto\int_V f^*\omega\cdot Td(V)
\]
is a bounded linear functional. Therefore since $\ZZ$ is closed in $\RR$, $\Omega_{BU}$ (resp. $\Omega_U$) is closed in $\Omega_{cl}^{ev}$ (resp. $\Omega_{cl}^{od}$). In particular, $\Omega^{od}/\Omega_U$ is Hausdorff.
\begin{remark}
	By replacing Thom's theorem with the Conner-Floyd theorem \cite{CF66} in the proof of \Cref{exactforms}, one can prove that a closed form is exact if and only if it has vanishing periods over all smooth SAC cycles.
\end{remark}

\begin{proposition}\label{prop:bottomexact}
The sequence
\[
H^{od}(\RR)\xrightarrow{\deR} \Omega^{od}/\Omega_U\xrightarrow{d} \Omega_{BU}\xrightarrow{\deR} H^{ev}(\RR)
\]
is strictly exact.
\end{proposition}

We will see shortly (at the beginning of the next subsection) that the kernel of $H^{od}(\RR)\xrightarrow{\deR} \Omega^{od}/\Omega_U$ is a lattice $L_U$ of full rank (i.e. the rank of $L_U$ is the dimension of $H^{od}(\RR)$). Similarly the image of $\Omega_{BU}\xrightarrow{\deR} H^{ev}(\RR)$ is a lattice $L_{BU}$ of full rank. Granted these, we can proceed to prove our proposition.
\begin{proof}
We know already from \cite{SS10} this sequence is exact, it remains to show each map is continuous and strict.

(1) $H^{od}(\RR)\xrightarrow{\deR} \Omega^{od}/\Omega_U$ is continuous since it is the composition of the continuous maps $H^{od}(\RR)=\Omega_{cl}^{od}/d\Omega^{ev}\hookrightarrow \Omega^{od}/d\Omega^{ev}$ and $\Omega^{od}/d\Omega^{ev}\twoheadrightarrow \Omega^{od}/\Omega_U$. Therefore, it induces an continuous bijection $H^{od}(\RR)/L_U\to \im(\deR)$. Notice that $H^{od}(\RR)/L_U$ is a finite dimensional torus which is in particular compact, and $\im(\deR)$ as a subspace of the Hausdorff space $\Omega^{od}/\Omega_U$ is Hausdorff. Recall that a continuous bijection from a compact space to a Hausdorff space must be a homeomorphism, we conclude $H^{od}(\RR)/L_U\to \im\deR$ is an isomorphism of topological groups. So by \Cref{lem:strictprop} $H^{od}(\RR)\xrightarrow{\deR} \Omega^{od}/\Omega_U$ is strict.

(2) $\Omega_{BU}\xrightarrow{\deR} H^{ev}(\RR)$ is continuous since it is the composition of the continuous maps $\Omega_{BU}\hookrightarrow \Omega^{ev}_{cl}$ and $\Omega^{ev}_{cl}\twoheadrightarrow \Omega^{ev}_{cl}/d\Omega^{od}=H^{ev}(\RR)$. Meanwhile its image $L_{BU}$ is discrete, therefore $\Omega_{BU}\xrightarrow{\deR} H^{ev}(\RR)$ is trivially strict.

(3) Finally we show $\Omega^{od}/\Omega_U\xrightarrow{d} \Omega_{BU}$ is continuous and strict. By \Cref{lem:strictprop} and definition of quotient topology, it suffices to show $d:\Omega^{od}\to \Omega_{BU}$ is continuous and strict, which follows from the continuity of $d$ and the openness of $d:\Omega^{od}\to d\Omega^{od}$ guaranteed by the open mapping theorem for Fr\'echet spaces.
\end{proof}

\subsection{Complex K-group} The (even) complex K-group $K(M)=K^0(M)$ is a finitely generated abelian group, we equip it with the discrete topology so that all the maps out of $K(M)$ are automatically continuous. Meanwhile a map into $K(M)$ is continuous if and only if for all $x\in K(M)$ the preimage of $x$ is clopen.

Recall that the Chern character map $K\to H^{ev}(\QQ)$ is an isomorphism when tensored with $\QQ$. Hence the rational K-theory $K_{\QQ}$ is isomorphic to $H^{ev}(\QQ)$. Also the image of the Chern character map is a discrete integral lattice $L_{BU}$ in $H^{ev}(\QQ)$ whose rank is the same as the dimension of $H^{ev}(\QQ)$. From the character diagram, this lattice $L_{BU}$ is also the image of $\Omega_{BU}\xrightarrow{\deR} H^{ev}(\RR)$. Similar analysis applies word-by-word to the Chern character map for the odd $K$-group $K^{-1}\to H^{od}(\QQ)$. We denote the corresponding lattice in $H^{od}(\QQ)\subset H^{od}(\RR)$ by $L_U$.

Now consider the coefficient long exact sequences for K-theory induced by $0\to \ZZ\to \QQ\to\QQ/\ZZ\to 0$ and $0\to \ZZ\to \RR\to\RR/\ZZ\to 0$, and identify the rational (resp. real) K-groups with the rational (resp. real) cohomology groups by means of Chern character, we then obtain the following commutative diagram with exact rows:
\begin{center}
	\begin{tikzcd}\label{coefficientseq}
		K^{-1}\ar[r]\ar[equal]{d} & H^{od}(\QQ)\ar[r]\ar[d,hook] & K^{-1}_{\QQ/\ZZ} \ar[r]\ar[d] & K\ar[r]\ar[equal]{d} & H^{ev}(\QQ)\ar[d,hook]\\
		K^{-1}\ar[r] & H^{od}(\RR)\ar[r,"\text{mod }\ZZ"] & K^{-1}_{\RR/\ZZ} \ar[r,"\beta"] & K\ar[r,"c"] & H^{ev}(\RR)
	\end{tikzcd}
\end{center}
From this commutative diagram, we can deduce several useful facts.
\begin{lemma}\label{lem:coefficientseq}
\begin{enumerate}[(1)]
	\item $K^{-1}_{\QQ/\ZZ}\to K^{-1}_{\RR/\ZZ}$ is an injection, and it identifies $K^{-1}_{\QQ/\ZZ}$ with $\Tor(K^{-1}_{\RR/\ZZ})$.
	\item The image of the Bockstein $\beta: K^{-1}_{\RR/\ZZ}\to K$ is $\Tor(K)$.
	\item $\text{mod }\ZZ: H^{od}(\RR)\to K^{-1}_{\RR/\ZZ}$ factors as $H^{od}(\RR)\twoheadrightarrow H^{od}(\RR)/L_U\hookrightarrow K^{-1}_{\RR/\ZZ}$, where the first map is a universal covering map.
\end{enumerate}
\end{lemma}
\begin{proof}
\begin{enumerate}[(1)]
	\item The injectivity follows from the five-lemma, we hence view $K^{-1}_{\QQ/\ZZ}$ as a subgroup of $K^{-1}_{\RR/\ZZ}$. Now tensoring the first row with $\QQ$ yields an exact sequence
	$$K^{-1}\otimes\QQ\to H^{od}(\QQ)\to K^{-1}_{\QQ/\ZZ}\otimes\QQ\to K\otimes\QQ\to H^{ev}(\QQ),$$
	where the first and last maps are isomorphisms. Hence $K^{-1}_{\QQ/\ZZ}\otimes\QQ=0$, that is to say $K^{-1}_{\QQ/\ZZ}$ is a torsion group. Meanwhile, if $x\in K^{-1}_{\RR/\ZZ}$ is a torsion, then $\beta(x)\in K$ is also a torsion, hence it is mapped to zero by $K\to H^{ev}(\QQ)$. By the exactness of the first row, we see that $x$ in fact belongs to $K^{-1}_{\QQ/\ZZ}$. This proves $K^{-1}_{\QQ/\ZZ}=\Tor(K^{-1}_{\RR/\ZZ})$.
	\item Since $\im \beta=\ker c$ and $H^{ev}(\RR)$ is torsion-free, we have $\Tor(K)\subseteq \im \beta$. On the other hand, $\im \beta$ coincides with the image of $K^{-1}_{\QQ/\ZZ}\to K$ by an easy diagram tracing. But we know from (1) that $K^{-1}_{\QQ/\ZZ}$ is a torsion group, so $\im\beta\subseteq \Tor(K)$.
	\item The factorization follows from the exactness of the bottom row. Also the map $H^{od}(\RR)\to H^{od}(\RR)/L_U$	, being the quotient map of a finite dimensional real vector space by a lattice of maximal rank, is clearly a universal covering map. More precisely, it is the universal covering map of a finite dimensional torus.
\end{enumerate}
\end{proof}

\begin{remark}
	Using the coefficient exact sequence associated to $0\to \ZZ\xrightarrow{\times n}\ZZ\to \ZZ/n\to 0$, the same argument shows that the image of $K^{-1}_{\ZZ/n}\to K^{-1}_{\RR/\ZZ}$ (induced by $\ZZ/n\hookrightarrow \RR/\ZZ$) is the subgroup consisting of $n$-torsions of $K^{-1}_{\RR/\ZZ}$.
	\begin{center}
	\begin{tikzcd}
 		K^{-1}_{\ZZ/n} \ar[r]\ar[d] & K\ar[r,"\times n"]\ar[equal]{d} & K\ar[d]\\
		K^{-1}_{\RR/\ZZ} \ar[r,"\beta"] & K\ar[r,"c"] & K_\RR
	\end{tikzcd}
\end{center}
\end{remark}

\begin{corollary}\label{cor:bottomright}
	$d\Omega^{od}$ is the identity component of $\Omega_{BU}$, and $\pi_0 (\Omega_{BU})\cong L_{BU}$.
\end{corollary}
\begin{proof}
	By \Cref{prop:bottomexact}, $\Omega_{BU}/d\Omega^{od}$ is homeomorphic to the image of $\Omega_{BU}\xrightarrow{\deR}H^{od}(\RR)$, which we have seen is $L_{BU}$. Since $L_{BU}$ is discrete, we conclude that $d\Omega^{od}$ is clopen in $\Omega_{BU}$. Now $d\Omega^{od}$, being a Fr\'echet space, is (arcwise) connected, so $d\Omega^{od}$ must be the identity component of $\Omega_{BU}$. It follows that $\pi_0 (\Omega_{BU})=\Omega_{BU}/d\Omega^{od}\cong L_{BU}$.
\end{proof}

\subsection{Odd K-group with $\RR/\ZZ$-coefficients} From \Cref{lem:coefficientseq} we have the following induced commutative diagram in which each row is exact.
\begin{center}
	\begin{tikzcd}
		0\ar[r] & H^{od}(\QQ)/L_U\ar[r]\ar[d,hook] & K^{-1}_{\QQ/\ZZ}\ar[r]\ar[d,hook] & \Tor(K)\ar[equal]{d}\ar[r] & 0\\
		0\ar[r] & H^{od}(\RR)/L_U\ar[r,"\text{mod }\ZZ"] & K^{-1}_{\RR/\ZZ}\ar[r,"\beta"] & \Tor(K)\ar[r] & 0
	\end{tikzcd}
\end{center}

\begin{proposition}\label{prop:topexact}
There is a unique topology on $K^{-1}_{\RR/\ZZ}$ making it into an abelian topological group such that
\[
0\to H^{od}(\RR)/L_U\xrightarrow{\text{mod }\ZZ} K^{-1}_{\RR/\ZZ}\xrightarrow{\beta}\Tor(K)\to 0
\]
is strictly exact. Moreover, under such topology $K^{-1}_{\QQ/\ZZ}$ is dense in $K^{-1}_{\RR/\ZZ}$.
\end{proposition}
\begin{proof}
Indeed, in order so, $\text{mod }\ZZ$ must induce a homeomorphism between $H^{od}(\RR)/L_U$ and $\beta^{-1}(0)$. In particular $\beta^{-1}(0)$ must be connected. Also $\beta$ must be a quotient map, so for each $x\in \Tor(K)$, $\beta^{-1}(x)$ should be clopen in $K^{-1}_{\RR/\ZZ}$. Moreover since the topology is required to be compatible with the group structure, $\beta^{-1}(x)$ should be homeomorphic to $\beta^{-1}(0)$ by translation. Therefore $K^{-1}_{\RR/\ZZ}=\coprod\limits_{x\in K} \beta^{-1}(x)$ must be the partition of $K^{-1}_{\RR/\ZZ}$ into its connected components, and each component is by translation homeomorphic to the identity component $\beta^{-1}(0)$. This in turn completely determines the topology on $K^{-1}_{\RR/\ZZ}$ and this topology clearly makes the above sequence strictly exact.

In order to check the density of $K^{-1}_{\QQ/\ZZ}$ in $K^{-1}_{\RR/\ZZ}$, it suffices to check that the intersection of $K^{-1}_{\QQ/\ZZ}$ with each connected component of $K^{-1}_{\RR/\ZZ}$ is dense in that component. This, by the above diagram and by translation homeomorphism, is reduced to verifying the density of $H^{od}(\QQ)/L_U$ in $H^{od}(\RR)/L_U$ which is now obvious.
\end{proof}

Therefore we equip $K^{-1}_{\RR/\ZZ}$ with the topology discussed in the above proof.

\subsection{Differential K-group}\label{subsec:diffK} By the same argument used in the proof of \Cref{prop:topexact}, we have
\begin{proposition}\label{prop:toponKhat}
There is a unique topology on the differential K-group $\hat{K}(M)$ compatible with its group structure such that the sequence
\[
0\to \Omega^{od}/\Omega_U \xrightarrow{i} \hat{K}\xrightarrow{\delta} K\to 0
\]
is strictly exact. Moreover, with such topology $\Omega^{od}/\Omega_U$ is the identity component of $\hat{K}$. \qed
\end{proposition}

Therefore we equip $\hat{K}$ with such unique topology as in \Cref{prop:toponKhat}. Note that $\Omega^{od}$ is locally connected, hence so are $\Omega^{od}/\Omega_U$ and $\hat{K}$. 

\begin{proposition} The sequence
\[
0\to K^{-1}_{\RR/\ZZ}\xrightarrow{j}\hat{K}\xrightarrow{ch}\Omega_{BU}\to 0
\]
is strictly exact. Moreover, $j$ maps $K^{-1}_{\RR/\ZZ}$ isomorphically onto $\overline{\Tor}(\hat{K})$.
\end{proposition}
\begin{proof}
First we show that $j$ is continuous and strict. Since $j$ is a group homomorphism and connected components are open in $K^{-1}_{\RR/\ZZ}$, it suffices to prove $j$ restricted to the identity component of $K^{-1}_{\RR/\ZZ}$ is continuous and strict. From the proof of \Cref{prop:topexact} and the character diagram, such restriction coincides with the composition $H^{od}(\RR)/L_U\xrightarrow{\deR}\Omega^{od}/\Omega_U\xrightarrow{i}\hat{K}$, whose continuity and strictness follow from the continuity and strictness of $\deR$ and $i$.

Similarly, to prove $ch$ is continuous and strict, we only need to check $ch$ restricted to the identity component of $\hat{K}$ is continuous and strict. Such restriction, by \Cref{prop:toponKhat} and the character diagram, coincides with $\Omega^{od}/\Omega_U\xrightarrow{d} \Omega_{BU}$, which we have seen is continuous and strict.

Lastly we prove that $\im j=\overline{\Tor}(\hat{K})$. Since $\Omega_{BU}$, being a subgroup of the torsion-free group $\Omega^{ev}$, is torsion-free, we see $\Tor(\hat{K})\subseteq \ker ch=\im j$. Now that $ch$ is continuous and $\Omega_{BU}$ is Hausdorff, we know $\im j=\ker j$ is closed, hence $\overline{\Tor}(\hat{K})\subseteq \im j$. Meanwhile we have $j(\Tor(K^{-1}_{\RR/\ZZ}))\subseteq \Tor(\hat{K})$, so by density of  $\Tor(K^{-1}_{\RR/\ZZ})=K^{-1}_{\QQ/\ZZ}$ in $K^{-1}_{\RR/\ZZ}$ we get $\im j\subseteq \overline{\Tor}(\hat{K})$.
\end{proof}

\subsection{Character diagram as hexagon diagram}
In summary, we have the following theorem, from which it should be clear that the character diagram coincides with the hexagon diagram for $\hat{K}$.
\begin{theorem}\label{webdiagram} There is a unique topology on $\hat{K}$, functorial with respect to smooth mappings, so that all the exact sequences in the character diagram are strictly exact. With such topology, we have the following commutative diagram with strictly exact rows and columns.
\begin{center}
\begin{tikzcd}
            & 0     \arrow[d]           & 0 \arrow[d]             & 0 \arrow[d]  &   \\
0 \arrow[r] & H^{od}(\RR)/L_U \arrow[r,"\text{mod }\ZZ"] \arrow[d,"\deR"] & K^{-1}_{\RR/\ZZ} \arrow[r,"\beta"] \arrow[d,"j"] & \Tor{K} \arrow[r] \arrow[d] & 0 \\
0 \arrow[r] & \Omega^{od}/\Omega_U \arrow[r,"i"] \arrow[d,"d"] & \hat{K} \arrow[r,"\delta"] \arrow[d,"ch"] & K \arrow[r] \arrow[d,"c"] & 0 \\
0 \arrow[r] & d\Omega^{od} \arrow[r] \arrow[d] & \Omega_{BU}\arrow[r,"\deR"] \arrow[d] & L_{BU} \arrow[r] \arrow[d] & 0 \\
            & 0                         & 0                       & 0           &  
\end{tikzcd}
\end{center}
Moreover each row is isomorphic to the strictly exact sequence associated to the identity component, and each column is isomorphic to the strictly exact sequence associated to closure of torsion.\qed
\end{theorem}

We end this section with two short exact sequences derived from the character diagram. They will be used in Section 3, in constructing universal classes and in proving the uniqueness. Consider the following push-out and the pull-back diagrams:
    \begin{center}
        \begin{tikzcd}
            H^{od}(\RR)/L_U \arrow[r,"\text{mod }\ZZ"] \arrow[d,"\deR"] & K^{-1}_{\RR/\ZZ}\ar[d] & \Omega_{BU}\times_{H^{ev}(\RR)} K \arrow[r] \arrow[d] & K \arrow[d,"c"]\\
            \Omega^{od}/\Omega_U \arrow[r] & \frac{\Omega^{od}/\Omega_U\times K^{-1}_{\RR/\ZZ}}{H^{od}(\RR)} & \Omega_{BU}\arrow[r,"\deR"] & L_{BU}
        \end{tikzcd}
    \end{center}
\begin{corollary}\label{cor:shortseqfromchardiag}
We have short exact sequences
    \[
    0\to \frac{\Omega^{od}/\Omega_U\times K^{-1}_{\RR/\ZZ}}{H^{od}(\RR)}\xrightarrow{i-j} \hat{K}\xrightarrow{c\circ\delta} L_{BU}\to 0,
    \]
    and
    \[
    0\to H^{od}(\RR)/L_U\xrightarrow{i\circ\deR} \hat{K}\xrightarrow{ch\times \delta} \Omega_{BU}\times_{H^{ev}(\RR)} K\to 0.
    \]
\end{corollary}
\begin{proof}
Exactness follows from a straightforward diagram tracing. Continuity is obvious and strictness follows from \Cref{lem:strictprop}.
\end{proof}


\section{Technical preparations}\label{sec4}
 The results in this section are slightly modified from those in \cite{BS10} to serve our purposes.
\subsection{Approximation of spaces by manifolds}
The following proposition generalizes \cite[Proposition 2.3]{BS10} by allowing the fundamental group to be finite. This generalization is not only necessary for the purpose of this paper, but also should allow us to handle KO-theory as well (recall that $\pi_1 BO=\ZZ/2$) without leaving the category of compact manifolds.\footnote{This is also pointed out in Bunke-Schick survey.}
\begin{proposition}\label{prop:approxibymfld}
Let $E$ be a pointed connected topological space. Assume that $\pi_1 E$ is finite and $\pi_k E$ is finitely generated for $k\ge 2$. Then there exists a sequence of pointed compact manifolds with boundary $\{E_k\}_{k\in \NN}$ together with embedding of manifolds $\mu_k: E_k\to E_{k+1}$ and continuous maps $\lambda_k:E_k\to E$ compatible with the embeddings, that is $\lambda_k=\lambda_{k+1}\circ\mu_k$, such that
\begin{enumerate}[(1)]
    \item $E_k$ is homotopy equivalent to a $k$-dimensional finite CW-complex;
    \item $\lambda_k$ is $k$-connected, i.e. $(\lambda_k)_*:\pi_* E_k\to\pi_* E$ is an isomorphism for $*<k$ and onto for $*=k$.
\end{enumerate}
\end{proposition}
\begin{proof}
The proof is essentially the same standard argument used in \cite[Proposition 2.3]{BS10}. One constructs $E_1$ as a thickening of a wedge sum of circles, each of which corresponds to an element of $\pi_1 E$. Then one can inductively construct $E_{k+1}$ from $E_{k}$ by thickening the latter in high dimensional Euclidean space and then
\begin{enumerate}[(i)]
	\item attach handles to annihilate the kernel of the epimorphism $(\lambda_{k})_*:\pi_{k} E_{k}\to \pi_{k}E$; and
	\item attach thickened spheres via connected sum to generate $\pi_{k+1} E$.
\end{enumerate}
We note that in step (i) we attach only finitely many handles because the kernel is finitely generated. Indeed, for $k=1$ the kernel is a finite index subgroup of a finitely generated group and therefore finitely generated; meanwhile for $k\ge 2$ since $E_{k}$ is a finite CW complex with a finite fundamental group, all its higher homotopy groups are finitely generated abelian groups due to a theorem of Serre, and consequently the kernel as a subgroup of a finitely generated abelian group is itself finitely generated. It is also clear that in step (ii) only finitely many thickened spheres are needed. Thus the $E_k$'s constructed in this way are compact.
\end{proof}

For the proof of the following Mittag-Leffler property, we refer the reader to \cite[Lemma 2.4]{BS10}.
\begin{corollary}\label{Mittag-Leffler}
Let $(E_k, \mu_k, \lambda_k)$ be as in \Cref{prop:approxibymfld}, then as subgroups of $H^s(E_k;A)$ we have
\[
\mu_k^{k+l,*} H^s(E_{k+l};A)=\mu_k^* H^s(E_{k+1};A)
\]
for all $k,l\ge 1$, $s\ge 0$ and abelian group $A$, where $\mu_k^{k+l}=\mu_{k+l-1}\circ\cdots\circ\mu_k$. \qed
\end{corollary}
\subsection{Universal classes}
Henceforth we fix an approximation $\{E_k,\mu_k,\lambda_k\}$ of $BU$. Let $u\in K(BU)=[BU,\ZZ\times BU]$ be the class that corresponds to the inclusion of identity component of $\ZZ\times BU$. Note that $c(u)\in H^{ev}(BU;\RR)$ is known as the universal Chern character whose degree $0$ part is $0$ as $u$ is of virtual rank $0$.

The proof of the next two propositions are similar to \cite{BS10}. The construction of the universal forms and universal differential K-theory classes are obstruction theoretical in nature: one inductively constructs the desired objects and, if necessary, goes one step back to modify the constructed ones for further extensions.
\begin{proposition}
There exists a sequence of differential forms $\omega_k\in \Omega_{BU}(E_k)$ so that $\deR(\omega_k)=\lambda_k^* c(u)$ and $\mu_k^*\omega_{k+1}=\omega_k$.
\end{proposition}
\begin{proof}
We inductively construct $\omega_k$. Set $\omega_0=0$. Suppose $\omega_{k-1}$ has been constructed. Since $E_k$ is homotopy equivalent to a finite CW-complex, the map $\lambda_k: E_k\to BU$ factors through $BU(n)$ for some big $n$, and thus pulls back a principal $U(n)$-bundle onto $E_k$. Enrich this bundle with a connection and let $\widetilde{\omega}_k$ be the Chern-Weil character form of the connection minus $n$. It is clear from construction that $\deR(\widetilde{\omega}_k)=\lambda_k^* c(u)$. Next, we compare $\mu_{k-1}^*\widetilde{\omega}_k$ to $\omega_{k-1}$. Now since
\begin{align*}
	\deR(\mu_{k-1}^*\widetilde{\omega}_k-\omega_{k-1})&=\mu_{k-1}^*\deR(\widetilde{\omega}_k)-\deR(\omega_{k-1})\\
	&= \mu_{k-1}^*\lambda_k^* c(u)-\lambda_{k-1}^* c(u)=0,
\end{align*}
there exists $\eta\in \Omega^{od}(E_{k-1})$ such that $\mu_{k-1}^*\widetilde{\omega}_k+d\eta=\omega_{k-1}$. Notice $\mu_{k-1}:E_{k-1}\to E_k$ is an embedding of manifolds, so we can extend $\eta$ to a form $\widetilde{\eta}$ on $E_k$ and define $\omega_k=\widetilde{\omega}_k+d\widetilde{\eta}$. It is straightforward to check $\omega_k$ satisfies the requirements.
\end{proof}

\begin{proposition}\label{prop:universalclass}
Let $(\hat{K},i,j,\delta,ch)$ be a differential K-functor. Then there exists a sequence $\hat{u}_k\in\hat{K}(E_k)$ so that $u_k:=\delta (\hat{u}_k)=\lambda_k^* (u)$, $ch(\hat{u}_k)=\omega_k$ and $\mu_k^*\hat{u}_{k+1}=\hat{u}_k$.
\end{proposition}
\begin{proof}
	Since $c(\lambda_k^* u)=\deR(\omega_k)$, by \Cref{cor:shortseqfromchardiag} we can independently for each $k$ find $\widetilde{u}_k\in \hat{K}(E_k)$ so that $\delta(\widetilde{u}_k)=\lambda_k^* u$ and $ch(\widetilde{u}_k)=\omega_k$. We now inductively modify $\widetilde{u}_k$ to $\hat{u}_k=\widetilde{u}_k+i\circ\deR(\rho_k)$ by appropriately choosing $\rho_k\in H^{od}(M;\RR)$ so that $\mu_k^*\hat{u}_{k+1}=\hat{u}_k$.
	
	Set $\hat{u}_0=0$ and $\hat{u}_1=\widetilde{u}_1$.  Assume by induction that we have found $\hat{u}_0,\dots,\hat{u}_{k-1}$ and $\hat{u}_k'$ as desired. Then by \Cref{cor:shortseqfromchardiag} $\mu_k^*\widetilde{u}_{k+1}-\hat{u}_k'=i\circ\deR(\rho)$ for some $\rho\in H^{od}(E_k;\RR)$. Thus $\mu_{k-1}^{k,*}\widetilde{u}_{k+1}-\hat{u}_{k-1}=i\circ\deR(\mu_{k-1}^*\rho)$. By \Cref{Mittag-Leffler}, we can find $\widetilde{\rho}\in H^{od}(E_{k+1};\RR)$ such that $\mu_{k-1}^{k+1,*}(\widetilde{\rho})=\mu_{k-1}^*\rho$. Now set $\hat{u}_k=\hat{u}_k'+i\circ\deR(\rho)-i\circ\deR(\mu_k^*\widetilde{\rho})$ and $\hat{u}_{k+1}'=\widetilde{u}_{k+1}-i\circ\deR(\widetilde{\rho})$. Then we have $\mu_{k-1}^*\hat{u}_k=\hat{u}_{k-1}$ and $\mu_k^*\hat{u}_{k+1}'=\hat{u}_k$. This completes the inductive step.
\end{proof}

\subsection{Chern-Simons integration formula}
\begin{proposition}\label{prop:cs}
Let $(\hat{K},i,j,\delta,ch)$ be a differential K-functor. Let $M$ be a compact manifold with corners and $\iota_0, \iota_1$ denote the inclusions of endpoints $M\hookrightarrow [0,1]\times M$. Then for any $\hat{x}\in \hat{K}([0,1]\times M)$, we have
\[
\iota_1^*\hat{x}-\iota_0^*\hat{x}=i\Big([\int_{[0,1]\times M/M} ch\hat{x}]\Big)
\]
where $\int_{[0,1]\times M/M}$ means integration over the fiber of the projection $p:[0,1]\times M\to M$.
\end{proposition}
\begin{proof}
	Since $\iota_0\circ p$ is homotopic to the identity map, we have $\delta(\hat{x}-p^*\iota_0^*\hat{x})=\delta\hat{x}-p^*\iota_0^*(\delta\hat{x})=0$ in $K([0,1]\times M)$. It follows that
	\[
	\hat{x}=p^*\iota_0^*\hat{x}+i([\omega])
	\]
	for some $\omega\in \Omega^{od}([0,1]\times M)$. Therefore $ch\hat{x}=ch(p^*\iota_0^*(\hat{x}))+ch\circ i([\omega])=p^*\iota_0^*(ch\hat{x})+d\omega$. Now write $\omega=\omega_0+dt\wedge \omega_1$ with $\omega_0(t),\omega_1(t)\in\Omega^*(M)$. Then
	\begin{align*}
		\frac{\partial}{\partial t}\lrcorner ch\hat{x}&= \frac{\partial}{\partial t}\lrcorner d\omega\quad(\text{since }p^*\iota_0^*(ch\hat{x})\text{ is independent of }t)\\
		&=\frac{\partial}{\partial t}\lrcorner (d\omega_0-dt\wedge d\omega_1)\\
		&=\partial_t\omega_0-d\omega_1,
	\end{align*}
	and
	\[
	\int_{[0,1]\times M/M} ch\hat{x}=\omega_0|_{t=1}-\omega_0|_{t=0}-d\big(\int_0^1\omega_1 dt\big).
	\]
	So we have
	\[
	i\Big(\big[\int_{[0,1]\times M/M}ch\hat{x}\big]\Big)=i\Big(\big[\iota_1^*\omega_0-\iota_0^*\omega_0\big]\Big).
	\]
	Meanwhile, again using $\hat{x}=p^*\iota_0^*\hat{x}+i([\omega])$ we have
	\begin{align*}
		\iota_1^*\hat{x}-\iota_0^*\hat{x}&=\iota_1^* p^*\iota_0^*\hat{x}+\iota_1^*i([\omega])-\iota_0^* p^*\iota_0^*\hat{x}-\iota_0^* i([\omega])\\
		&=i([\iota_1^*\omega-\iota_0^*\omega])\quad(\text{since }p\circ\iota_0=id_M=p\circ\iota_1)\\
		&= i\Big(\big[\int_{[0,1]\times M/M}ch\hat{x}\big]\Big).
	\end{align*}
\end{proof}

\section{Uniqueness of differential K-functor}\label{sec5}
We now prove

\begin{theorem}\label{maintheorem}
Any two differential K-functors $(\hat{K}',i',j',\delta',ch')$ and $(\hat{K},i,j,\delta,ch)$ are naturally equivalent via a natural transformation $\Phi:\hat{K}'\to \hat{K}$; such $\Phi$ is unique.
\end{theorem}

\subsection{Reduction to connected manifolds, virtual rank} Let $(\hat{K},i,j,\delta,ch)$ be a differential K-functor. By compactness, $M$ has only finitely many connected components. The differential K-group of $M$ is related to those of its connected components as follows:
\begin{lemma}
Let $M_1,\dots,M_r$ be all connected components of $M$. Then the product of maps $\hat{K}(M)\to\hat{K}(M_i)$ induced by the inclusions is an isomorphism $\hat{K}(M)\cong \hat{K}(M_1)\times\cdots\times \hat{K}(M_r)$ of abelian topological groups.
\end{lemma}
\begin{proof}
	This follows from that $\Omega^*(M)\cong\Omega^*(M_1)\times\cdots\times\Omega^*(M_r)$ and $K(M)\cong K(M_1)\times\cdots\times K(M_r)$ as abelian topological groups.
\end{proof}

So, in order to prove the uniqueness of differential K-functor, we are reduced to considering connected manifolds only. From now on, all manifolds are assumed to be connected.

\begin{definition}[virtual rank] For $\hat{x}\in \hat{K}(M)$ we define its virtual rank to be the virtual rank of $\delta(\hat{x})\in K(M)$.
\end{definition}
Since $\Omega^{od}(pt)=0$, we can deduce from the character diagram that $\delta:\hat{K}(pt)\to K(pt)$ is an isomorphism.
The mapping $M\to pt$ induces $K(pt)\to K(M)$. Denote the image of $n\in\ZZ\cong K(pt)$ by $\n$, and similarly define $\hat{\n}\in \hat{K}(M)$. The virtual ranks of $\n$ and $\hat{\n}$ are both $n$.

\subsection{Natural transformation $\Phi$} Let $(\hat{K}',i',j',\delta',ch')$ be another differential K-functor, and choose universal classes $\hat{u}_k\in \hat{K}(E_k)$, $\hat{u}_k'\in \hat{K}'(E_k)$ as in \Cref{prop:universalclass}.

For a manifold $M$ and $\hat{x}'\in\hat{K}'(M)$ of virtual rank $n$, we consider $\delta'(\hat{x}'-\hat{\n}')\in K(M)$. Since $\delta'(\hat{x}'-\hat{\n}')$ has virtual rank $0$, then by compactness of $M$ we can find a smooth mapping $f_k: M\to E_k$ so that $\delta'(\hat{x}')-\n=f_k^* u_k=\delta'(f_k^*\hat{u}_k')$. Therefore there is a unique $[\eta]\in \Omega^{od}/L_U(M)$ such that
\[
\hat{x}'=\hat{\n}'+f_k^*\hat{u}_k'+i'([\eta]).
\]
Define $\Phi: \hat{K}'(M)\to \hat{K}(M)$ by
\[
\hat{x}=\Phi(\hat{x}'):=\hat{\n}+f_k^*\hat{u}_k+i([\eta]).
\]
\begin{lemma}
$\Phi$ is well-defined.
\end{lemma}
\begin{proof}
    The only choice made is the mapping $f_k$. Two such choices are smooth homotopic if we pass to a bigger $k$. Suppose we have two such mappings $f_{k,i}: M\to E_k$ for $i=0,1$ and between which a smooth homotopy $F_k:[0,1]\times M\to E_k$. Then by \Cref{prop:cs} for $\hat{K}'$, we have
    \[
    f_{k,1}^*\hat{u}_k'-f_{k,0}^*\hat{u}_k'=i'\Big(\big[\int_{[0,1]\times M/M} F_k^*\omega_k\big]\Big).
    \]
    Then by construction we find $\eta_0,\eta_1\in\Omega^{od}(M)$ so that
    \begin{align*}
        \hat{x}'&=\hat{\n}'+f_{k,0}^*\hat{u}_k'+i'([\eta_0])\\
        &=\hat{\n}'+f_{k,1}^*\hat{u}_k'+i'([\eta_1]).
    \end{align*}
    It follows that
    \[
    i'([\eta_1-\eta_0])=i'\Big(\big[\int_{[0,1]\times M/M} F_k^*\omega_k\big]\Big).
    \]
    Now that $i'$ is injective, we see
    \[
    [\eta_1-\eta_0]=\big[\int_{[0,1]\times M/M} F_k^*\omega_k\big].
    \]
    Then applying \Cref{prop:cs} to $\hat{K}$ we have
    \[
    f_{k,1}^*\hat{u}_k-f_{k,0}^*\hat{u}_k=i\Big(\big[\int_{[0,1]\times M/M} F_k^*\omega_k\big]\Big)=i([\eta_1-\eta_0]).
    \]
    This proves $\Phi$ is well-defined.
\end{proof}
\begin{lemma}\label{lem:commuttativity}
$\Phi$ is natural and $\Phi\circ i=i'$, $\delta\circ\Phi=\delta'$, $ch\circ\Phi=ch'$.
\end{lemma}
\begin{proof}
    Straightforward verifications.
\end{proof}
\begin{proposition}\label{prop:grouphom} 
$\Phi$ is a group homomorphism.
\end{proposition}
\begin{proof}
Consider the deviation of $\Phi$ from being a group homomorphism, that is, a natural transformation $$\widetilde{B}:\hat{K}'\times\hat{K}'\to \hat{K}$$ such that $$\Phi(\hat{x}'+\hat{y}')=\Phi(\hat{x}')+\Phi(\hat{y}')+\widetilde{B}(\hat{x}',\hat{y}').$$
Clearly $\widetilde{B}(\hat{x}',\hat{y}')=\widetilde{B}(\hat{y}',\hat{x}')$. Also from our definition and \Cref{lem:commuttativity} we have
\begin{align*}
    0&= \Phi(\hat{x}'+i'([\eta]))-i([\eta])-\Phi(\hat{x}')=\widetilde{B}(\hat{x}',i'([\eta]))\\
    0&= \delta(\Phi(\hat{x}'+\hat{y}')-\Phi(\hat{x}')-\Phi(\hat{y}'))=\delta(\widetilde{B}(\hat{x}',\hat{y}'))\\
    0&= ch(\Phi(\hat{x}'+\hat{y}')-\Phi(\hat{x}')-\Phi(\hat{y}'))=ch(\widetilde{B}(\hat{x}',\hat{y}')).
\end{align*}
It follows $\widetilde{B}$ factors over a natural transformation $$B: K(M)\times K(M)\to H^{od}(M;\RR).$$ Then by that every finite CW-complex is homotopy equivalent to a compact manifold with boundary, every continuous map between manifolds is homotopic to a smooth map, Adams' variant of Brown representation theorem \cite{Ad71} and the Yoneda lemma, we see $B$ is represented by the homotopy class of some map $$(\ZZ\times BU)\times (\ZZ\times BU)\to \prod_{i=1}^\infty K(\RR,2i-1)$$ which must be homotopy equivalent to the trivial map since $H^{od}((\ZZ\times BU)\times (\ZZ\times BU);\RR)=0$. This proves the natural transformation $B=0$ and therefore $\widetilde{B}=0$.
\end{proof}
\begin{corollary}
$\Phi$ is continuous and strict.
\end{corollary}
\begin{proof}
    This follows from that $\Phi$ is a group homomorphism and that $\Phi$ restricted to identity component is continuous and strict.
\end{proof}
\begin{corollary}
$\Phi$ is an isomorphism of abelian topological groups.
\end{corollary}
\begin{proof}
	Apply the above corollary and five lemma to the following commutative diagram:
	\[
	\begin{tikzcd}
		0 \ar[r]&\Omega^{od}/\Omega_U\ar[equal]{d}\ar[r,"i'"] & \hat{K}'\ar[d,"\Phi"]\ar[r,"\delta' "]& K\ar[equal]{d}\ar[r] & 0\\
		0 \ar[r]&\Omega^{od}/\Omega_U\ar[r,"i"] & \hat{K}\ar[r,"\delta"]& K\ar[r] & 0
	\end{tikzcd}
	\]
\end{proof}
\subsection{The proof of $\Phi\circ j'=j$}\label{rigidity of j}
It remains to prove $\Phi\circ j'=j$. By continuity and density of $K^{-1}_{\QQ/\ZZ}$ in $K^{-1}_{\RR/\ZZ}$, it suffices to show that $\Phi\circ j'|_{K^{-1}_{\QQ/\ZZ}}=j|_{K^{-1}_{\QQ/\ZZ}}$. 


Since $K^{-1}_{\QQ/\ZZ}=\varinjlim\limits_n K^{-1}_{\ZZ/n}$, we consider the composition
\[
j_{\ZZ/n}: K^{-1}_{\ZZ/n}\to K^{-1}_{\QQ/\ZZ}\xrightarrow{j} \hat{K}
\]
and similarly define $j_{\ZZ/n}': K^{-1}_{\ZZ/n}\to \hat{K}'$. It suffices to show $\Phi\circ j_{\ZZ/n}'= j_{\ZZ/n}$ for all $n$.

Let $\GL(\ZZ/n)$ represent the functor $K^{-1}_{\ZZ/n}$ (i.e. the degree $-1$ space in the Omega-spectrum representing the cohomology theory $K^{*}_{\ZZ/n}$). Then we have
\[
\pi_* \GL(\ZZ/n)=
\begin{cases}
\ZZ/n \quad& *\text{ is odd},\\
0\quad &*\text{ is even}.
\end{cases}
\]
We caution the reader that $\GL(\ZZ/n)$ is \textit{not} the general linear group over the ring $\ZZ/n$.

By \Cref{prop:approxibymfld}, we can find an approximation $\{F_k,\xi_k,\theta_k\}$ of $\GL(\ZZ/n)$ by compact manifolds-with-boundary.
\begin{center}
        \begin{tikzcd}
        F_k\ar[rr,"\theta_k"]\ar[rd,"\xi_k"']& &F_{k+1}\ar[ld,"\xi_{k+1}"]\\
        & \GL(\ZZ/n) &
        \end{tikzcd}
    \end{center}
Let $v\in K^{-1}_{\ZZ/n}(\GL(\ZZ/n))$ be the class corresponding to the identity map of $\GL(\ZZ/n)$, and denote $v_k=\xi_k^* v\in K^{-1}_{\ZZ/n}(F_k)$.

For each $\rho\in K^{-1}_{\ZZ/n}(M)$, we can find a smooth mapping $g_k: M\to F_k$ so that $\rho=g_k^* v_k$. Then by functoriality, the desired identity $\Phi\circ j_{\ZZ/n}'(\rho)=j_{\ZZ/n}(\rho)$ will follow from:
\begin{proposition}
$\Phi\circ j_{\ZZ/n}'(v_k)=j_{\ZZ/n}(v_k)$.
\end{proposition}
\begin{proof}
Notice $\delta\circ\Phi\circ j_{\ZZ/n}'=\delta'\circ j_{\ZZ/n}'$ coincides with $\delta\circ j_{\ZZ/n}$; both being the composition $K^{-1}_{\ZZ/n}\to K^{-1}_{\RR/\ZZ}\xrightarrow{\beta} K$. Hence $\Phi\circ j_{\ZZ/n}'(v_k)-j_{\ZZ/n}(v_k)\in\ker \delta$. Meanwhile, since $ch\circ\Phi\circ j_{\ZZ/n}'=ch'\circ j_{\ZZ/n}'=0=ch\circ j_{\ZZ/n}$, we see $\Phi\circ j_{\ZZ/n}'(v_k)-j_{\ZZ/n}(v_k)\in\ker ch$. Then from \Cref{cor:shortseqfromchardiag} we can find $\sigma_k\in H^{od}(F_k;\RR)$ so that $$i\circ\deR(\sigma_k)=\Phi\circ j_{\ZZ/n}'(v_k)-j_{\ZZ/n}(v_k).$$ Similarly, for all $l\ge 1$, we can find $\sigma_{k+l}\in H^{od}(F_{k+l};\RR)$ such that $$i\circ\deR(\sigma_{k+l})=\Phi\circ j_{\ZZ/n}'(v_{k+l})-j_{\ZZ/n}(v_{k+l}).$$
Therefore, we have
\[
i\circ\deR(\theta_k^{k+l,*}\sigma_{k+l})=i\circ\deR(\sigma_k).
\]
By injectivity of $i$, we see $\deR(\theta_k^{k+l,*}\sigma_{k+l})=\deR(\sigma_k)$. Now by construction $\xi_{k+l}: F_{k+l}\to \GL(\ZZ/n)$ is $(k+l)$-connected, it follows for $s\le k$ we have
$$H^s(F_{k+l};\RR)\cong H^{s}(\GL(\ZZ/n);\RR)=0.$$
Combined with $H^s(F_k;\RR)=0$ for $s\ge {k+1}$ (since $F_k$ is homotopy equivalent to a $k$-dimensional CW complex), we have $\deR(\sigma_k)=0$. This completes the proof for $\Phi\circ j_{\ZZ/n}'(v_k)=j_{\ZZ/n}(v_k)$.
\end{proof}

\subsection{Uniqueness of $\Phi$}
Finally we argue that the natural transformation $\Phi:\hat{K}'\to \hat{K}$ is unique. From the exact sequences in \Cref{cor:shortseqfromchardiag}, the (additive) difference between any two such natural transformations factors through a natural transformation from $L_{BU}$ into $H^{od}(\RR)/L_U$. Now we precompose this natural transformation with the surjective Chern character map $c: K\to L_{BU}$; the uniqueness of $\Phi$ will follow from that any natural transformation $K\to H^{od}(\RR)/L_U$ is zero. Indeed this follows, by employing the approximation-by-manifolds argument as before, from that the functor $K$ is represented by $\ZZ\times BU$ whose odd real cohomology vanishes.

\section{Rigidity of topological theories}\label{sec6}
We show that the topological theories $K^{-1}_{\RR/\ZZ}$ and $K$ are rigid relative to differential K-theory in the sense below. The first statement shows that there is a unique way of realizing $K^{-1}_{\RR/\ZZ}$ as the flat theory in the sense of \cite{BS10}, the second statement generalizes \cite[Corollary 1.1]{SS08}.
\begin{theorem}\label{rigidity}
	For a differential K-functor, $(\hat{K},i,j,\delta, ch)$,
	\begin{enumerate}[(1)]
		\item $j$ is uniquely determined by $i,\delta,ch$.
		\item $\delta$ is uniquely determined by $i,j,ch$.
	\end{enumerate}
\end{theorem}
\begin{proof}
\begin{enumerate}[(1)]
	\item It follows from \Cref{webdiagram} by an easy diagram tracing that the difference between any two $j$'s factors through a natural transformation $\Tor K\to H^{od}(\RR)/L_U$. To see this must be zero, we precompose it with the surjective Bockstein transformation $K^{-1}_{\QQ/\ZZ}\to \Tor K$, and then note that any natural transformation from $K^{-1}_{\QQ/\ZZ}$ into $H^{od}(\RR)/L_U$ must be zero following the same argument as in \Cref{rigidity of j}.
	\item Similarly we have that the difference between any two $\delta$'s factors through a natural transformation $\rho: \Omega_{BU}\to \Tor K$. Now $\rho$ restricted to $d \Omega^{od}$ must be zero, since $d\Omega^{od}$ is a real vector space so must its image, but $\Tor K$ is a finitely generated abelian group. Therefore $\rho$ descends to a natural transformation $L_{BU}\to \Tor K$, which must be zero. This can be seen by precomposing it with the surjective transformation $K\to L_{BU}$ and noticing that $K(\ZZ\times BU)$ is torsion-free which forces any natural transformation $K\to \Tor K$ to be zero using the approximation-by-manifolds argument.
\end{enumerate}
\end{proof}
\begin{remark}
	The first part of the theorem above also follows from applying \Cref{maintheorem} to $(\hat{K},i,j,\delta,ch)$ and $(\hat{K},i,j',\delta,ch)$. One simply observes that in this case $\Phi$ is the identity map by construction. 
\end{remark}

\bibliographystyle{alpha}
\bibliography{ref}
\end{document}